%% file: preconditioner_paper.tex
\documentclass[twoside]{article}
\usepackage[accepted]{aistats2018}
\usepackage{times}
\usepackage{epsfig}
\usepackage{graphicx}
\usepackage{amsmath}
\usepackage{amssymb}
\usepackage{amsthm}
\usepackage{subfig}
\usepackage{algorithm}
\usepackage{algorithmic}
\usepackage[numbers]{natbib}
\usepackage{hyperref}
\usepackage{xcolor}
\usepackage{makecell}

\input{sections/shortcut}

\graphicspath{{./figures/}}
\begin{document}

%

%

\twocolumn[


\aistatstitle{Combinatorial Preconditioners for Proximal Algorithms on Graphs}

\aistatsauthor{ Thomas M\"ollenhoff \And Zhenzhang Ye \And Tao Wu \And Daniel Cremers}

\aistatsaddress{ Department of Informatics, Technical University of Munich, Germany\\ \{thomas.moellenhoff, zhenzhang.ye, tao.wu, cremers\}$@$\hspace{0.3mm}tum.de }
]

\begin{abstract}
We present a novel preconditioning technique for proximal optimization methods
that relies on graph algorithms to construct effective preconditioners.
Such combinatorial preconditioners arise from partitioning 
the graph into forests. We prove that certain decompositions lead to 
a theoretically optimal condition number. \tr{We also show how ideal decompositions 
can be realized using matroid partitioning and propose efficient greedy variants thereof 
for large-scale problems.}
Coupled with specialized solvers for the
resulting scaled proximal subproblems, the preconditioned algorithm
achieves competitive performance in machine learning and vision applications.
\end{abstract}

\input{sections/introduction.tex}
\input{sections/pdhg.tex}

\input{sections/preconditioner.tex}

\input{sections/numerics.tex}

\input{sections/conclusion.tex}

{
\bibliographystyle{abbrvnat}
\bibliography{sections/biblio}
}

\end{document}

%% file: sections/shortcut.tex
\newtheorem{theorem}{Theorem}
\newtheorem{corollary}[theorem]{Corollary}
\newtheorem{lemma}[theorem]{Lemma}
\newtheorem{proposition}[theorem]{Proposition}
\newtheorem{definition}[theorem]{Definition}

\newcommand{\bthm}{\begin{theorem}}
\newcommand{\ethm}{\end{theorem}}
\newcommand{\bcor}{\begin{corollary}}
\newcommand{\ecor}{\end{corollary}}
\newcommand{\blem}{\begin{lemma}}
\newcommand{\elem}{\end{lemma}}
\newcommand{\bprop}{\begin{proposition}}
\newcommand{\eprop}{\end{proposition}}
\newcommand{\bdefn}{\begin{definition}}
\newcommand{\edefn}{\end{definition}}
\newcommand{\bpf}{\begin{proof}}
\newcommand{\epf}{\end{proof}}

\newcommand{\ienu}[1]{\begin{enumerate}#1\end{enumerate}}

\newcommand{\iali}[1]{\begin{align}#1\end{align}}
\newcommand{\ialid}[1]{\begin{aligned}#1\end{aligned}}
\newcommand{\ieqn}[1]{\begin{equation}#1\end{equation}}

\newcommand{\tr}{} 
\newcommand{\tb}{} 
\newcommand{\tg}{} 

\DeclareMathOperator{\rnk}{rank}

\DeclareMathOperator{\krn}{ker}
\DeclareMathOperator{\ran}{ran}

\DeclareMathOperator{\diag}{diag}

\newcommand{\bR}{\mathbb{R}}

\newcommand{\cE}{\mathcal{E}}
\newcommand{\cT}{\mathcal{T}}

\newcommand{\cI}{\mathcal{I}}

\newcommand{\cP}{\mathcal{P}}

\newcommand{\cV}{\mathcal{V}}

\newcommand{\cG}{\mathcal{G}}

\renewcommand{\hat}{\widehat}

\newcommand{\norm}[1]{\| #1 \|}
\newcommand{\ip}[2]{\left\langle#1,#2\right\rangle}

\newcommand{\cart}{\,\square\,}
\newcommand{\etal}{{et~al.}}

%% file: sections/introduction.tex
\section{Introduction}

Many applications in statistics \cite{tibshirani2005sparsity}, learning \cite{hein2011beyond}, and imaging \cite{ChPo11} rely on efficiently solving convex-concave saddle-point problems:
\iali{
\max_{p\in\bR^\cE} \min_{u\in\bR^\cV} ~ G(u)-F^*(p)+\ip{Ku}{p}.
\label{eq:spp}
}
Here the model is formulated on \tr{an undirected weighted} graph $\cG=(\cV,\cE,\omega)$, whose edges are weighted by a given function $\omega: \cE\to\bR_+$. 
The extended real-valued functions $F:\bR^\cE\to\bR\cup\{+\infty\}$ and $G:\bR^\cV\to\bR\cup\{+\infty\}$ are assumed to be proper, lower semi-continuous and convex. The notation $F^*$ refers to the convex conjugate of $F$. The (linear) vertex-to-edge map $K:\bR^\cV\to\bR^\cE$ is defined by
\ieqn{
K=\diag(\omega)\nabla, \notag
\label{eq:wgrad}
} 
where $\nabla$ is the (transposed) incidence matrix of $\cG$, i.e.
\ieqn{
(\nabla u)_e = u_{i}-u_{j}, \quad \forall e=(i,j)\in\cE, \notag
\label{eq:wtv}
}
\tr{with arbitrarily fixed orientation}.
For $F$ being the $\ell^1$-norm, i.e.~$F(\cdot)=\|\cdot\|_1$, this choice yields the total-variation semi-norm of functions on a weighted graph. \tg{In addition to} its ubiquitous applications in image processing and computer vision, this semi-norm has recently gained considerable attention in unsupervised learning~\cite{hein2011beyond,hein2013total,bresson2013multiclass}, semi-supervised learning~\cite{garcia2014multiclass}, collaborative filtering~\cite{benzi2016song}, clustering~\cite{garcia2014multiclass} and statistical inference~\cite{xin2014efficient}.

Among other proximal algorithms (see \cite{PaBo13,ChPo16b} and the references therein for an overview), the primal-dual hybrid gradient (PDHG) algorithm \cite{AHU58, ZhCh08, PCBC-ICCV09, EZC10, ChPo11} is a popular solver for the problem in \eqref{eq:spp}. 
A general formulation of PDHG iterations \cite{ChPo16b} appears as follows:
\iali{
\hspace{-0.1cm}&u^{k+1} = \arg\min_{u \in \bR^\cV}~ G(u) + \langle{p^k},{Ku}\rangle + \frac{s}{2}\norm{u-u^k}_S^2, \label{eq:ppdhg1}\\ 
&p^{k+1} = \arg\min_{p \in \bR^\cE} ~ F^*(p) - \ip{K(2u^{k+1}-u^k)}{p} \notag\\
 &\qquad\qquad +\frac{t}{2}\norm{p-p^k}_T^2. \label{eq:ppdhg2}
}
Here $S\in\bR^{|\cV|\times|\cV|}$ and $T\in\bR^{|\cE|\times|\cE|}$ are \tr{symmetric} positive definite matrices, \tg{such} that $\|\cdot\|_S$ is a scaled norm defined by $\|u\|_S^2={\ip{u}{u}_S}=u ^\top Su$ and analogously for $\|\cdot\|_T$. For given $S$ and $T$, the convergence of PDHG is guaranteed if the \tr{(inverse)} step sizes $s,t$ satisfy $s t > \norm{T^{-1/2} K S^{-1/2}}^2$, cf.~\cite[Lemma 1]{PoCh11}.
Interestingly, as pointed out by \cite{EZC10,ChPo11,ChPo16b}, the formulation in \eqref{eq:ppdhg1}--\eqref{eq:ppdhg2} provides a flexible framework \tg{for} deriving various types of proximal splitting algorithms, e.g.~the proximal gradient method, \tr{Douglas-Rachford splitting,} and (linearized) ADMM. 
Proximal algorithms in form of \eqref{eq:ppdhg1}--\eqref{eq:ppdhg2} are particularly \tr{efficient} when $F$, $G$ admit {\it separable} structures and $S$, $T$ are {\it diagonal}. In this scenario, solutions of the subproblems in \eqref{eq:ppdhg1} and \eqref{eq:ppdhg2} refer to pointwise proximal evaluations. 
Nonetheless, many research efforts have been devoted to further accelerating the convergence speed.

To this end, two categories of acceleration strategies are envisaged: multi-step acceleration and preconditioning.
Classical (optimal) multi-step gradient descent methods are attributed to \cite{Pol64,Nes83}.
In the context of proximal algorithms, the FISTA algorithm \cite{BeTe09} was proposed as an accelerated proximal gradient method, and in \tb{\cite{ChPo16a}} a multi-step PDHG was devised. 

On the other hand, a preconditioning technique aims to accelerate convergence, typically through reducing the number of outer iterations, by choosing proper scaling matrices $S$ and $T$ (also called the preconditioners in this context).
In contrast to their counterparts for solving linear systems, preconditioning techniques for proximal algorithms are \tg{much} less developed. 
Diagonal preconditioners for PDHG were explored in \tb{\cite{PoCh11}}; Preconditioners for other types of proximal algorithms \cite{BeFa12,GiBo14a,GiBo14b,LSS14,ZYDR14,BrSu15,GiBo15,FoBo15} also appeared recently in the \tr{literature}.

A consensus among all existing preconditioning approaches appears that, while favorably reducing the number of outer iterations, preconditioning could explode the overall computational load by expensive inner proximal \tr{evaluations}. For example, for proximal gradient method in minimizing a \tr{sum} of smooth \tr{and} nonsmooth functions, using the (approximate) Hessian of the smooth function as preconditioner may attain superlinear convergence in the outer iteration but \tg{also} lead to very expensive proximal (or backward) steps; see \cite{LSS14}. This is the reason why diagonal preconditioners remain a popular choice in many recent works, and non-diagonal preconditioners designed in \cite{BeFa12,ZYDR14,FrGo16} do not deviate far from the diagonal ones.

In this work, we propose combinatorial preconditioners for proximal algorithms based \tb{on} a partitioning of the original graph into forests. This leads to a class of block diagonal preconditioners, and the resulting PDHG updates refer to solving parallel subproblems on forests. We show how to construct such preconditioners guided by theoretical \tg{estimates} of \tg{the} condition number. Coupled with fast direct solvers for proximal evaluation on forests, we achieve significant performance boost for the \tb{PDHG} algorithm across a series of numerical tests.

%% file: sections/pdhg.tex
\section{Preconditioner and Condition Number}
\label{sec:pdhg}

The choice of $S$ and $T$ can significantly influence the convergence speed of the (generalized) PDHG scheme, \eqref{eq:ppdhg1}--\eqref{eq:ppdhg2}, in practice. In \cite{PoCh11}, Pock and Chambolle showed that \tr{utilization of} diagonal preconditioners $S$ and $T$ yields \tg{a} visible performance boost in comparison with PDHG without preconditioning (i.e.~$S=I,~T=I$). In a slightly different context, Boyd et al.~\cite{DiBo17, FoBo15, GiBo14a,GiBo14b,GiBo15} also considered diagonal preconditioning strategies for other closely related proximal algorithms. 
In particular, they suggested based on extensive numerical experiments that an ideal choice of $S$ and $T$ ought to minimize the \emph{(finite) condition number}
$\kappa(T^{-1/2}KS^{-1/2})$ defined by
\begin{equation} 
\kappa(\cdot) := \frac{\sigma_{\max}(\cdot)}{\sigma_{\min>0}(\cdot)},
\label{eq:cond_num}
\end{equation}
i.e.~the ratio between largest and smallest non-zero singular value. This rule of thumb was computationally pursued by so called matrix equilibration \cite{DiBo17, GiBo15, Zhu17}.

A more quantified connection between the convergence rate of \eqref{eq:ppdhg1}--\eqref{eq:ppdhg2} and $\kappa(T^{-1/2}KS^{-1/2})$ can be drawn in a more specific setting, \tb{e.g.~$G(u) = \frac{1}{2}\norm{u-f}^2$} for some given $f\in\bR^\cV$.
By choosing $s=1,~S=I$ in \eqref{eq:ppdhg1}--\eqref{eq:ppdhg2}, one \tr{comes} up with the following proximal gradient iteration:
\iali{
  p^{k+1} =\, &\arg\min_{p\in\bR^\cE} F^*(p) \notag\\ &
  +\frac t2\norm{p- \tb{\left( p^k - (tT)^{-1} K (K^\top p^k - f) \right)} }_T^2.
  \label{eq:proxgrad}
}

\tr{When} $F^* = \delta_{\tr{Q}}$ is the indicator function of a polyhedral set $\tr{Q} := \{ p \in \bR^{\cE} ~:~ Ap \leq b \}$ it can be shown that iteration \eqref{eq:proxgrad} converges linearly
\cite{necoara2015linear}. More precisely, with step size choice $t = \sigma_{\text{max}}(T^{-1/2}K)$ the linear convergence rate $r \in [0, 1)$ can be derived (cf.~\cite[Theorem~11]{necoara2015linear}) as:
\begin{equation} \notag
  r = \frac{\kappa^2 - 1}{\kappa^2 + 1}, ~~ \kappa = \theta(T^{-1/2}A, T^{-1/2} K)\sigma_{\max}(T^{-1/2} K),
\end{equation}
where $\theta(\cdot, \cdot)$ denotes Hoffman's bound \cite{Hoffman52,wang2014}. 
In the extreme case where $F^*\approx0$ \tb{(e.g., for strong regularizations)}, Hoffman's bound reduces to $\sigma_{\min>0}(T^{-1/2} K)^{-1}$ and \tr{$\kappa$} matches the condition number $\kappa(T^{-1/2} K)$. 

\tb{In a general setting}, a strong correlation between the convergence speed and the condition number is supported by \tb{the numerical evidences in} Section~\ref{sec:ROF} and Table~\ref{table:ROF_synthetic}.

%% file: sections/preconditioner.tex
\section{Combinatorial Preconditioners}
Following the discussion in the previous section, reducing $\kappa(T^{-1/2} KS^{-1/2})$ provides a reasonable guideline for the choice of preconditioners.
As $K$ is the weighted incidence matrix of $\cG$, this boils down to constructing ``good" approximations of the graph.
Meanwhile, as discussed in the introduction, the non-diagonal $S$ or $T$ could possibly explode the computational cost of the proximal evaluation in \eqref{eq:ppdhg1}--\eqref{eq:ppdhg2}. Hence, an ideal choice of preconditioners would strike a balance between reduced (outer) iteration number and \tg{costlier proximal evaluations} \tr{per} iteration. 

Bearing this in mind, in this section we propose our combinatorial preconditioners based on graph partitioning. 
This leads to a family of block diagonal preconditioners for $T$. For simplicity we \tg{fix} $S=I$ in our development, and remark that $T$ can be adapted to work properly with any diagonal preconditioner $S$ (e.g.~\tg{the one} used in \cite{PoCh11}). 
In terms of the update scheme \eqref{eq:ppdhg2}, such preconditioning yields proximal evaluation on respective partitioned subgraphs, which can be efficiently carried out by using {the state-of-the-art direct solver on trees \cite{KPR16}}.

As a remark,
there is a connection between our combinatorial preconditioners and the subgraph preconditioners for solving linear systems in graph Laplacians; see, e.g., \cite{spielman2010algorithms} and the references therein. 
Pioneered by Vaidya \tr{and his coworkers} in \tg{the early 1990s} \cite{vaidya90,Joshi96}, a series of works are done in finding a subgraph preconditioner.
An ideal subgraph preconditioner uses a graph Laplacian on some (low-stretch) spanning tree which best preserves connectivity between vertices in the original graph. 
In some sense, our proposed combinatorial preconditioners are dual analogues of the subgraph preconditioners. 

In Section \ref{sec:subgph}, we show how to construct $T$ via graph partitioning and under which sufficient conditions the condition number $\kappa(T^{-1/2} K)$ can be (optimally) reduced. In Section \ref{sec:forest}, such sufficient conditions are algorithmically realized by: (1) chains on regular grid; (2) nested forests on general graph.
In Section \ref{sec:backward}, we detail the message passing based implementation of efficient proximal evaluation on forests.

\subsection{Preconditioning via Graph Partitioning} \label{sec:subgph}
Let the edge set $\cE$ be partitioned into $L$ mutually disjoint subsets, i.e.~$\cE=\bigsqcup_{l=1}^L\cE_l$, such that each subgraph $\cG_l=(\cV,\cE_l,\omega|_{\cE_l})$ is a {\it forest},
i.e.~$\cG_l$ has no cycle.
Correspondingly, we define $P_l$ as the canonical projection from $\bR^\cE$ to $\bR^{\cE_l}$, i.e.~$P_lp=p|_{\cE_l}$ for each $p\in\bR^\cE$.
Thus, the matrix $K$ can be decomposed into submatrices
$\{K_l\}_{l=1}^L$ 
where each $K_l=P_l K\in\bR^{|\cE_l|\times|\cV|}$. Analogously let $\nabla_l = P_l \nabla$ and $\omega_l = \omega|_{\cE_l}$. \tr{Note that each $\nabla_l^\top$ has full column rank.}

We then define our preconditioners as follows
\ieqn{ \label{eq:tmtx}
\ialid{
T_l &:= K_l K_l^\top \quad\forall l\in\{1,...,L\}, \\
T &:= \sum_{l=1}^L P_l^\top T_l P_l.  \notag
}}

It follows from \eqref{eq:cond_num} that
\iali{ \label{eq:lambdapi}
\kappa(T^{-1/2} K)
=\sqrt{\frac{\lambda_{\max}(\Pi)}{\lambda_{\min>0}(\Pi)}},
}
\ieqn{ \label{eq:sumproj}
\hspace{-0.15cm}
\ialid{
\Pi := K^\top T^{-1}K=\sum_{l=1}^L\Pi_l, ~
\Pi_l := \nabla_l^\top (\nabla_l\nabla_l^\top)^{-1} \nabla_l.
}}
Indeed, each $\Pi_l$ is the orthogonal projection onto the subspace $\ran \nabla_l^\top$,
and hence $\ran \Pi_l=\ran \nabla_l^\top$ and $\krn \Pi_l=\krn \nabla_l$.
It follows immediately from \eqref{eq:sumproj} that 
\ieqn{\label{eq:lmaxpi}
\lambda_{\max}(\Pi)\leq L,
}
and therefore it suffices for the sake of convergence to choose step sizes $s,t$ such that $st>L$.

In the remainder of Section \ref{sec:subgph}, we \tg{show analytically} how certain graph partitions can attain optimal condition number $\kappa(T^{-1/2}K)$ in a two-partition scenario, i.e.~$L=2$. 
The two-partition scenario is motivated from many applications where $\cG$ is a 2D regular grid (whose maximum degree equals 4). In this case, it is guaranteed by Nash-Williams' theorem \cite{Nas64} that $\cG$ can be covered by two disjoint forests.
Our theory suggests that either a {\it chain} decomposition or a {\it nested-forest} decomposition makes a good preconditioner on the 2D regular grid. Furthermore, the nested-forest preconditioners \tr{further extend} to any general weighted graph, with aid of either the matroid partition algorithm \cite{Edm65} or some greedy algorithm; see Section \ref{sec:nested_forest}.

The following theorem derives a lower bound for $\kappa(T^{\tr{-1/2}} K)$ for a wide range of two-partition cases. 
We denote by $\rnk \nabla^\top$ the column rank of $\nabla^\top$. In terms of graph theory, $\rnk \nabla^\top$ identifies the maximum number of the edges in $\cG$ that are cycle-free. 
\tr{$T^{-1}$ is understood as the Moore-Penrose pseudoinverse when $T$ is singular.
The two assumptions rule out pathological cases where $\cG$ is either too sparse (e.g.~cycle-free) or too dense (e.g.~fully connected).
In the proof, }
we shall use Weyl's inequality \cite{Wey12}: 
\iali{ \label{eq:weyl1}
\lambda_{i+j-1}(\Pi_1+\Pi_2)\leq \lambda_i(\Pi_1)+\lambda_j(\Pi_2),
}
\tb{where $i, j \geq 1$, $i + j - 1 \leq |\cV|$, and} $\lambda_i(\cdot)$ denotes the $i$-th largest eigenvalue of a real symmetric matrix.

\begin{theorem} \label{thm:lbnd}
  Let $\cG$ be partitioned into two nonempty subgraphs $\cG_1$ and $\cG_2$ (not necessarily forests) such that 
  \ienu{
  \item
  $\ran \nabla_1^\top  \cap \ran \nabla_2^\top \supsetneq \{ 0 \}.$
  \item
  $\rnk \nabla^\top >\min(\rnk \nabla_1^\top,\rnk \nabla_2^\top).$
  }
  Then we have
  $\kappa(T^{\tr{-1/2}} K) \geq \sqrt{2}$.
\end{theorem}
\begin{proof}
By Weyl's inequality \eqref{eq:weyl1}, we have $\lambda_1(\Pi)\leq\lambda_1(\Pi_1)+\lambda_1(\Pi_2)=2$. Moreover, condition 1 ensures that there exists some nonzero $v\in\bR^\cV$ with $\Pi v=\Pi_1 v+\Pi_2 v=2v$. Hence, we must have $\lambda_1(\Pi)=2$.

Let $r=\rnk \nabla^\top$. Without loss of generality, we assume $\rnk \nabla_1^\top\geq\rnk \nabla_2^\top$. Hence $r>\rnk \nabla_2^\top$ due to condition 2.
Again using \eqref{eq:weyl1}, we have $\lambda_{\min>0}(\Pi)=\lambda_r(\Pi)\leq \lambda_1(\Pi_1)+\lambda_r(\Pi_2)=1+0=1$. Altogether, the conclusion follows in view of \eqref{eq:lambdapi}.
\end{proof}

We proceed to propose sufficient conditions for graph partitioning (in particular when $L=2$) 
which guarantee optimal condition number \tb{$\kappa(T^{-1/2}K)=\sqrt{2}$}. 

\begin{theorem}
  Let $\cG$ be partitioned into two forests $\cG_1$ and $\cG_2$.
  If the following conditions are satisfied
  \begin{enumerate}
    \item $\Pi_1 \Pi_2 = \Pi_2 \Pi_1$, \label{item:commute}
    \item $\krn (\Pi_1 \Pi_2) \setminus \krn \Pi \neq \emptyset$, \label{item:kernel}
    \item $\ran \Pi_1  \cap \ran \Pi_2 \supsetneq \{ 0 \}$, \label{item:range}
  \end{enumerate}
  then $\kappa(T^{\tr{-1/2}} K) = \sqrt{2}$.
  \label{lem:condition}
\end{theorem}
\begin{proof}
The proof of $\lambda_{\max}(\Pi)=2$ is identical to \tb{that for} Theorem \ref{thm:lbnd}. 

On the other hand, we have 
\iali{
\lambda_{\min>0}(\Pi)
&=\min\{\ip{u}{\Pi u}/\|u\|^2: u\in \ran \Pi,~u\neq 0\}, \notag\\
&=\min\{\ip{v}{\Pi^3 v}/\|\Pi v\|^2: \Pi v\neq0\}, \notag
} 
where by commutativity (i.e.~condition 1)
\iali{
\frac{\ip{v}{\Pi^3 v}}{\|\Pi v\|^2}
&=\frac{\ip{v}{\Pi_1v}+\ip{v}{\Pi_2v}+6\ip{v}{\Pi_1\Pi_2v}}{\ip{v}{\Pi_1v}+\ip{v}{\Pi_2v}+2\ip{v}{\Pi_1\Pi_2v}} \notag\\
&=1+\frac{4\ip{v}{\Pi_1\Pi_2v}}{\ip{v}{\Pi_1v}+\ip{v}{\Pi_2v}+2\ip{v}{\Pi_1\Pi_2v}}. \notag
}
Note that $\Pi_1\Pi_2$ is an orthogonal projection onto $\ran \Pi_1  \cap \ran \Pi_2$, and hence $\ip{v}{\Pi_1\Pi_2v}\geq0$ for all $v$. Meanwhile, condition 2 asserts that $\Pi v\neq0$ and $\ip{v}{\Pi_1\Pi_2v}=0$ for some $v$. Altogether, $\lambda_{\min>0}(\Pi)=1$ and the conclusion follows.
\end{proof}

\begin{theorem}
  \label{thm:nested}
  Let $\cG$ be partitioned into $L$ (nonempty) nested forests, namely $\{ \cG_l \}_{l=1}^L$, in the sense that
  \iali{
    \ran \Pi_1 = ... = \ran \Pi_{\hat l} \supsetneq \ran \Pi_{\hat l+1}  \supseteq ... \supseteq \ran \Pi_L.  \notag
  }
  Then we have $\kappa(T^{\tr{-1/2}} K) = \sqrt{L/\hat l}$.
\end{theorem}
\begin{proof}
Recall \eqref{eq:lmaxpi} that $\lambda_{\max}(\Pi)\leq L$. Indeed, we have $\lambda_{\max}(\Pi)=L$ since $\Pi v=Lv$ for any nonzero $v\in \ran \Pi_L$. 
On the other hand, note that $\ran \Pi=\ran \Pi_1=... = \ran \Pi_{\hat l}$.
Therefore, for all $v\in \ran \Pi$, we have$\ip{v}{\Pi v}\geq \sum_{l=1}^{\hat l}\ip{v}{\Pi_l v}=\hat l\|v\|^2$ and the equality holds when $v\in \ran\Pi \setminus \ran \Pi_{\hat l+1}$. This gives \tb{$\lambda_{\min > 0}(\Pi)=\hat l$}, which concludes the proof.
\end{proof}

\subsection{Two Classes of Forest Preconditioners} \label{sec:forest}
\input{sections/forest}

\subsection{Proximal Evaluation on Forests} \label{sec:backward}
\input{sections/subproblem}

%% file: sections/forest.tex
\begin{figure*}
  \centering
  \begin{tabular}{cccc}
    \includegraphics[width=0.225\linewidth, trim={2cm 1.5cm 2cm 1.5cm},clip]{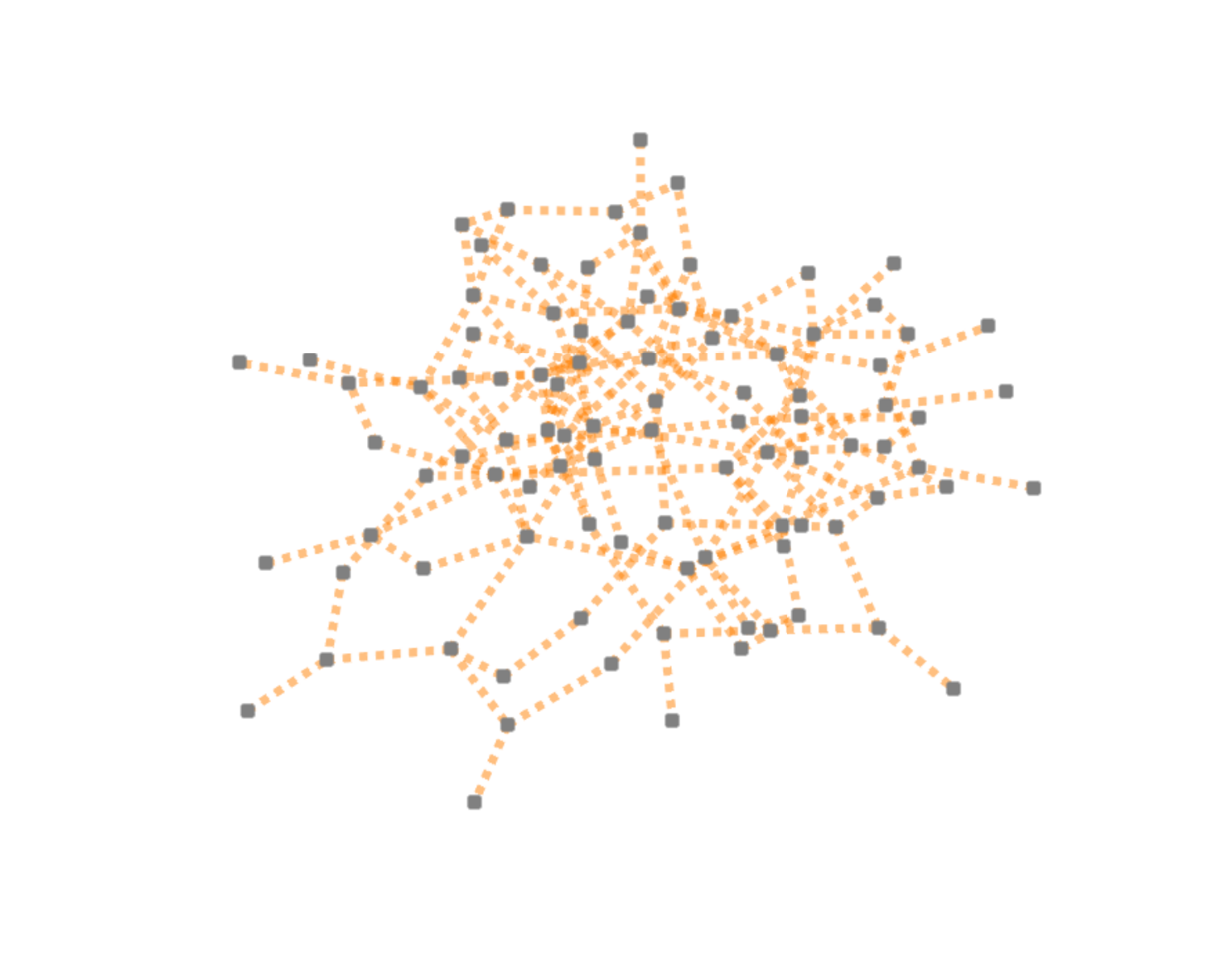}&
    \includegraphics[width=0.225\linewidth, trim={2cm 1.5cm 2cm 1.5cm},clip]{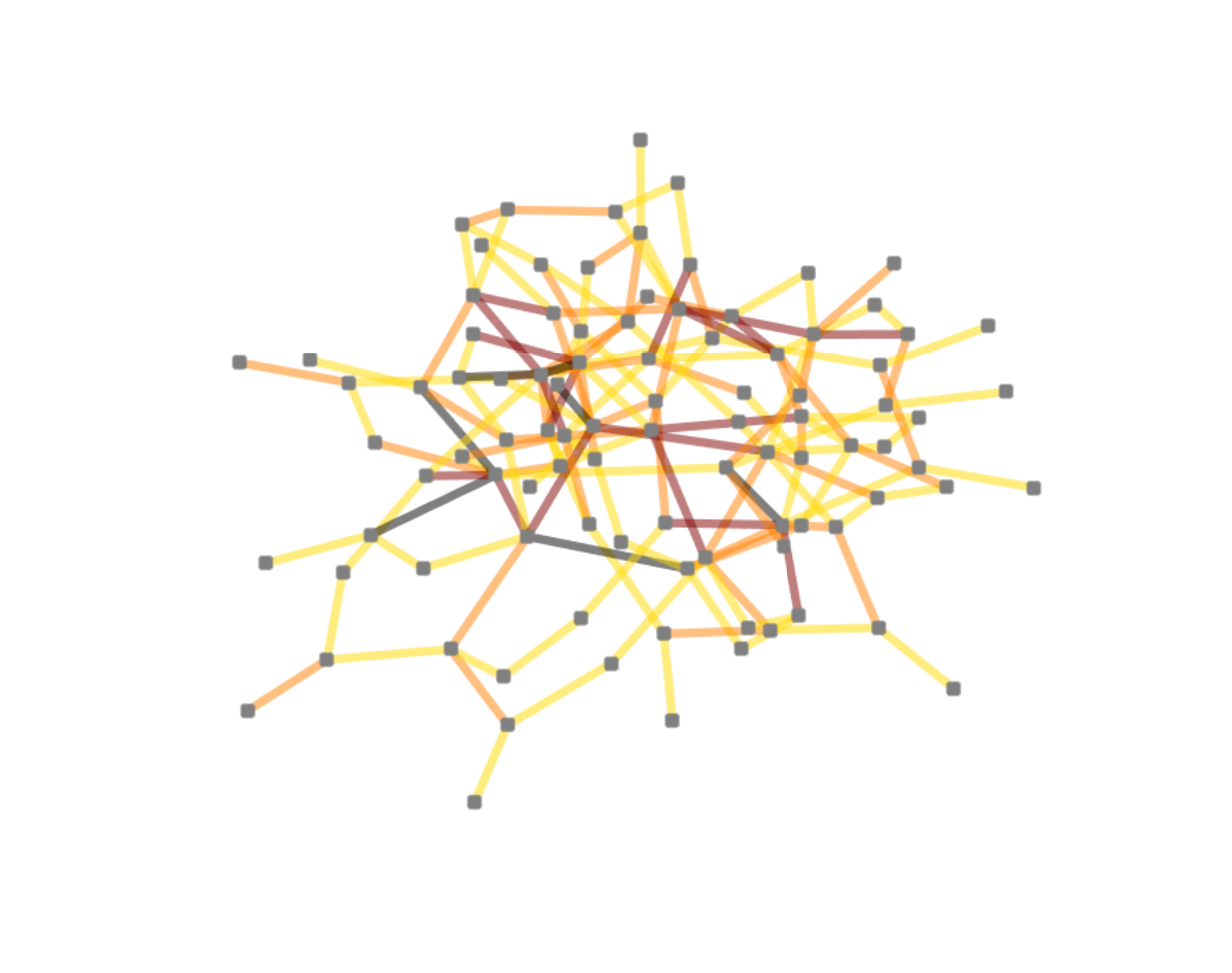} &
    \includegraphics[width=0.225\linewidth, trim={2cm 1.5cm 2cm 1.5cm},clip]{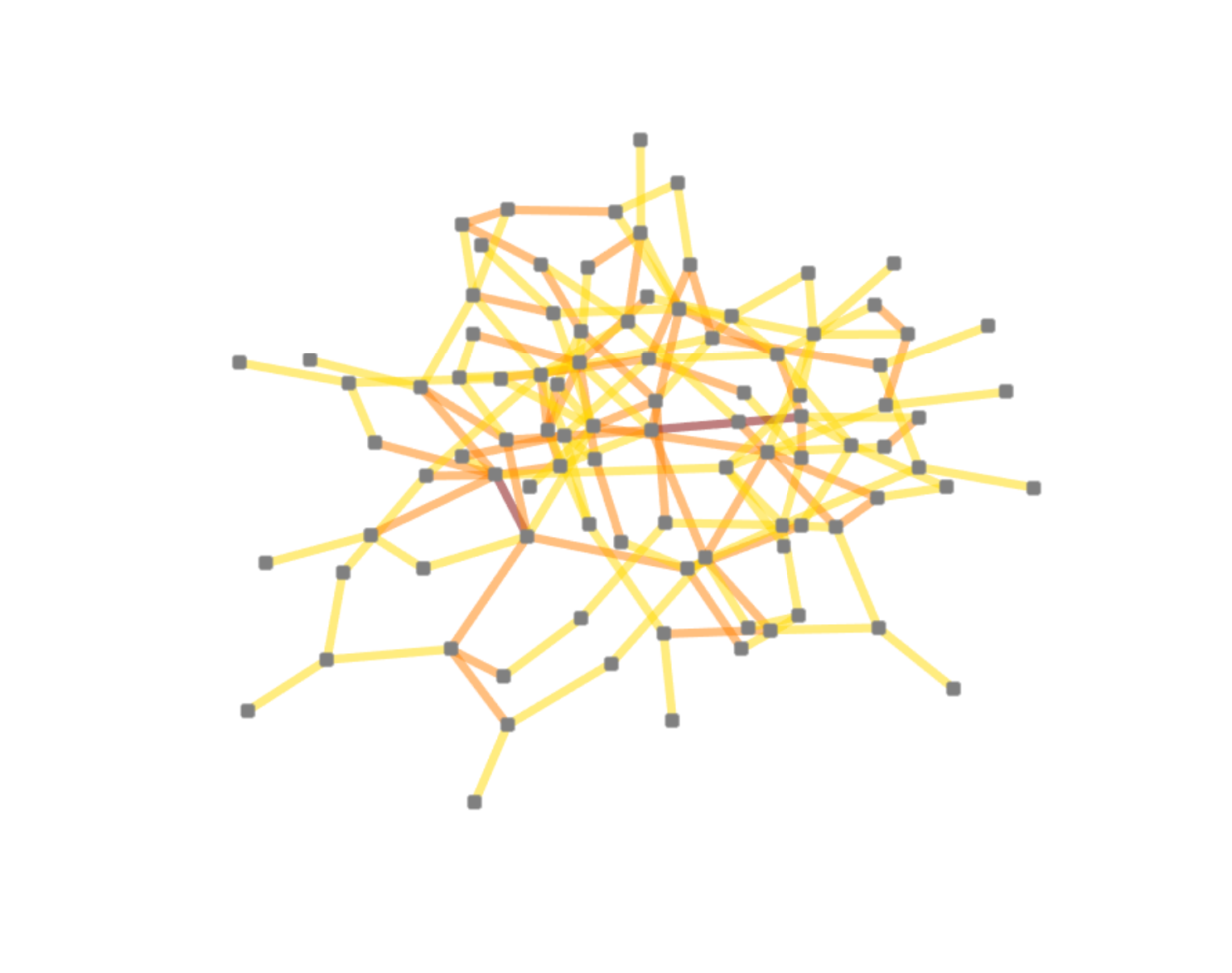}&
    \includegraphics[width=0.225\linewidth, trim={2cm 1.5cm 2cm 1.5cm},clip]{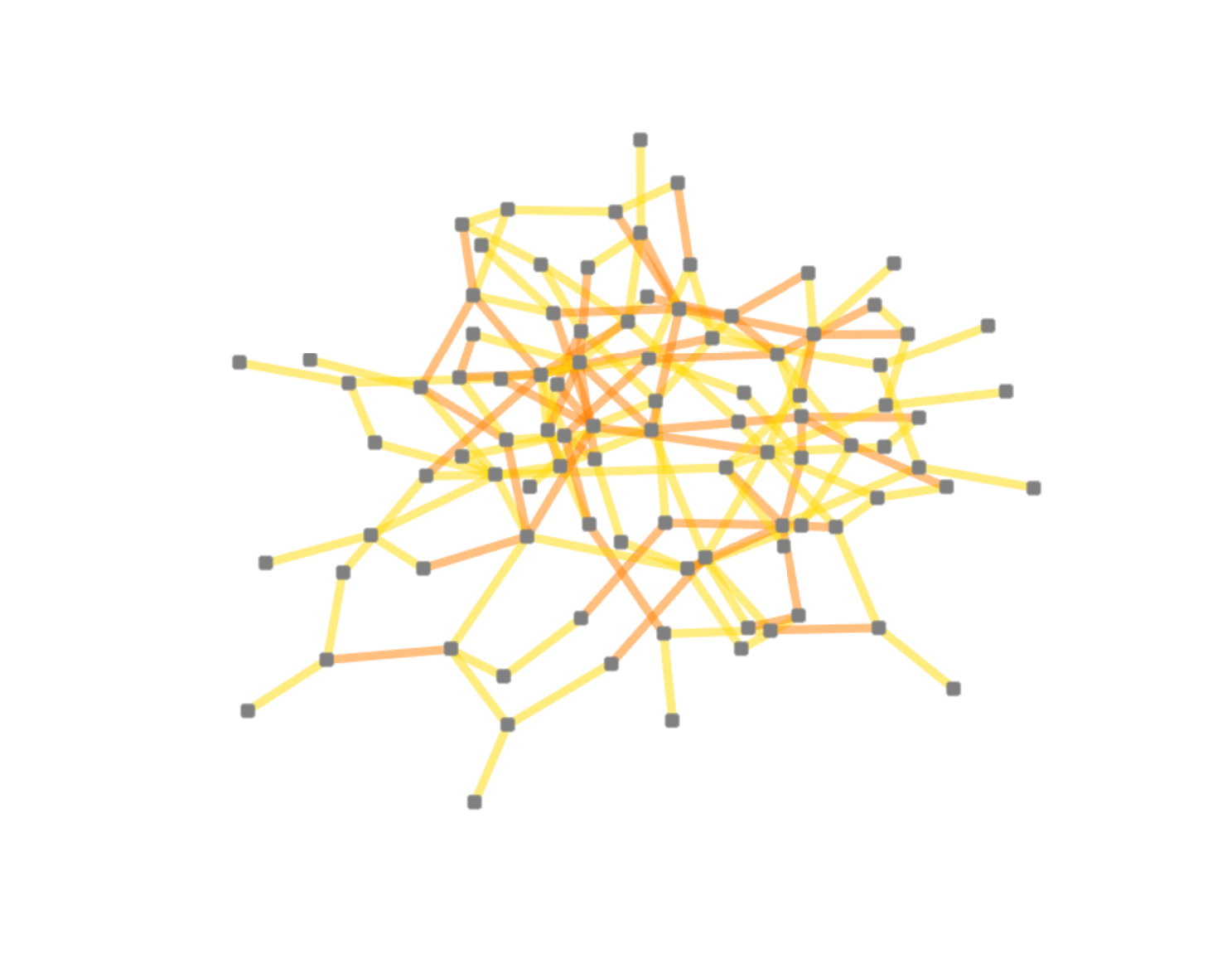}\\
    \makecell{No preconditioning\\ $\kappa(K) = 6.41$} & \makecell{Greedy linear forests\\$\kappa(T^{-1/2}K) = 5.13$} & \makecell{Greedy nested forests\\$\kappa(T^{-1/2}K) = \sqrt{3}$} & \makecell{Matroid partitioning\\$\kappa(T^{-1/2}K) = \sqrt{2}$} 
  \end{tabular}
  \caption{Overview of the proposed graph decompositions on an example graph ($|\cV|=96$, $|\cE|=150$). The decomposition into linear forests does not satisfy the assumptions of
    Theorem~\ref{thm:nested} and does not lead to a significant reduction in condition number. While the greedy nested forest decomposition satisfies the assumptions, it finds a partition into three
    spanning forests leading to a suboptimal condition number. The matroid approach guarantees the best possible condition number.}
\end{figure*}

Given the abstract conditions from Theorem~\ref{lem:condition}~and~\ref{thm:nested} under which graph partitioning achieves 
optimal condition number, the question remains whether such partitions exist and 
how they can be constructed in practice. In the following section we propose optimal
partitioning approaches, depending on the topology of the underlying graph. We
first focus on regular grid graphs, which are ubiquitous 
in signal processing and computer vision applications.

\subsubsection{Chains on Regular Grid}
\label{sec:forest_regular_grid}
Formally, $d$-dimensional grid graphs are given as the $d$-fold Cartesian product
$\cG = \cP_1 \cart \hdots \cart \cP_d$ between path graphs $\{ \cP_l \}_{l=1}^d$, 
$\cP_l = (\{1, \hdots, n_l\}, \{(i, i+1)\}_{i=1}^{n_l-1})$.
For ease of presentation we will focus mostly on two dimensional grids.
Recall that the Cartesian product $(\cV_1, \cE_1) \cart (\cV_2, \cE_2)$ yields a graph with vertex set 
$\cV_1 \times \cV_2$. In that graph, two vertices $(v_1, v_2)$ and $(v_1', v_2')$ are adjacent
if and only if $v_1 = v_1'$ and $v_2$ is adjacent \tg{to} $v_2'$ or $v_2 = v_2'$ and $v_1$ is adjacent \tg{to} $v_1'$.
Whether a given graph is a grid graph can be tested in linear time \cite{Imrich07}. 
The procedure described in \cite{Imrich07} yields a natural decomposition of the $d$-dimensional grid into 
$d$ \emph{linear forests}, i.e., a forest in which every tree is a path graph. Moreover, in the aforementioned decomposition, each forest consists of paths of equal length. To be precise, in the case $d=2$ the decomposition is given by 
\begin{equation}
  \begin{aligned}
    &\cG_1 = \cI_{n_2} \cart \cP_1, ~ \cG_2 = \cP_2 \cart \cI_{n_1}, 
  \end{aligned}
  \label{eq:chains}
\end{equation}
where $\cI_n = (\{ 1, \hdots, n\}, \emptyset)$. We refer to the connected components 
of $\cG_1$ as horizontal chains and $\cG_2$ as vertical chains.
Such a natural splitting into chains enjoys a theoretically optimal condition number among
the large class of decompositions characterized by Theorem~\ref{thm:lbnd}. 
\begin{theorem}
  \label{thm:cond_grid}
  Let $\cG = \cP_1 \cart \cP_2$ be a regular grid of dimension $n_1 \times n_2$ ($n_1,n_2\geq 3$). The condition 
  number in the unpreconditioned case is bounded by
  \begin{equation}
    \kappa(K) \geq \sqrt{2} \max ( n_1, n_2 ) / \sqrt{\pi^2 / 2}.
    \label{eq:grid_unprec}
  \end{equation}
  The natural partitioning of $\cG$ into two linear forests \eqref{eq:chains} 
  achieves condition number
  \begin{equation}
    \kappa \bigl( T^{-1/2} K \bigr) = \sqrt{2}.
    \label{eq:grid_prec}
  \end{equation}
\end{theorem}  
\begin{proof}
  The bound~\eqref{eq:grid_unprec} follows from elementary results in spectral graph theory. The spectrum $\{ \lambda_i \}$ of 
  a path graph of dimension $n$ is $\lambda_{i} = 2 - 2 \cos(\pi i / n)$, $i\in\{1,...,n-1\}$,
  see~\cite[Section~1.4.4]{brouwer12}. The spectrum of the Cartesian product graph is given by the summing 
  all pairs of eigenvalues of the individual graphs, cf.~\cite[Section~1.4.6]{brouwer12}.  
  Using the inequality $1 - \cos(x) \leq x^2 / 2$ we have
  \begin{equation} \notag
    \kappa(K)^2 = \frac{2 - \sum_{l=1}^2 \cos \left(\frac{\pi (n_l - 1)}{n_l} \right) }{1 - \cos \left( \frac{\pi}{\max(n_1, n_2)} \right) } \geq \frac{2 \max(n_1, n_2)^2}{ \pi^2 / 2}.
  \end{equation}
  
  To show \eqref{eq:grid_prec} it suffices to verify the three conditions in Theorem~\ref{lem:condition}. Denoting by $\mathbf{1}_n$ the ones-vector and by $\otimes$ the 
  Kronecker product, the projections are explicitly given by
  \begin{equation} \notag
    \begin{aligned}
&\Pi_1 = I_{n_2}\otimes \left(I_{n_1}-\frac{1}{n_1}(\mathbf{1}_{n_1}\mathbf{1}_{n_1}^\top)\right), \\
&\Pi_2 = \left(I_{n_2}-\frac{1}{n_2}(\mathbf{1}_{n_2}\mathbf{1}_{n_2}^\top)\right)\otimes I_{n_1}.
    \end{aligned}
  \end{equation}
  \tg{Since} $(A \otimes B)(C \otimes D) = (AC \otimes BD)$, the condition $\Pi_1 \Pi_2 = \Pi_2 \Pi_1$ from Theorem~\ref{lem:condition} \tg{follows} directly.
  
  Now note that each projection $\Pi_i$ subtracts the mean on the chains it contains. Conditions 2 and 3 can be verified by counting dimensions.
  For condition 2, simply pick nonzero $u \in \bR^\cV$ to be constant along chains in $\cP_1$ but non-constant along chains in $\cP_2$. For condition 3, pick nonzero $u \in \bR^\cV$ that is zero-mean on the chains in both $\cP_1$ and $\cP_2$.
\end{proof}
We remark that the above results can be generalized to the $d$-dimensional setting, yielding $\kappa(K) \geq \sqrt{d} (\max_i n_i) / \sqrt{\pi^2 / 2}$ and $\kappa \bigl( T^{-1/2} K \bigr) = \sqrt{d}$.
Remarkably the chains preconditioning \tg{mentioned above} makes the condition number independent of the grid size, while for the unpreconditioned case, it exhibits linear growth with respect to the largest grid dimension.

Furthermore, the splitting into chains leads to a particularly efficient evaluation of the dual subproblem in PDHG as we will see later on. Together with the theoretically optimal condition number, this makes \tg{the chains} the number one choice of preconditioner for the regular $d$-dimensional grid. 
\subsubsection{Nested Forests on General Graphs} \label{sec:nested_forest}
\paragraph{Matroid partitioning.}The situation on general graphs is more involved than on grid graphs.
Inspired by Theorem \ref{thm:nested}, we seek to partition $\cG$ into nested forests such that the number of forests is \emph{minimal}, i.e.~equal to the {\it arboricity} of the graph \cite{Nas64}. 
If $\cG$ is connected, the arboricity of $\cG$ can be calculated as:
\tr{$\max\left\{\lceil \frac{|\cE_{\cG'}|}{|\cV_{\cG'}|-1} \rceil: \cV_{\cG'}\subset\cV_{\cG},\,\cE_{\cG'}\subset\cE_{\cG} \right\}$.}

It turns out that the classical matroid partitioning algorithm by Edmonds \cite{Edm65} meets our \tg{requirements}. 
In short, matroid partitioning progressively inserts an idle edge into the forest partitions. To preserve all partitions cycle-free, it \tr{relies} on a primitive operation which detects a simple cycle (also called {\it circuit}) whenever this occurs due to the insertion of a new edge into a forest. By the nature of the algorithm, the resulting partitions are guaranteed to be (a minimal number of) nested forests, and hence realize the condition posed in Theorem \ref{thm:nested}.

\begin{table}
  \centering
  \resizebox{\linewidth}{!}{
  \begin{tabular}{c|cc|cc|cc|cc}
    & & &&& &\multicolumn{2}{c}{Nested Forest} \\
    & \multicolumn{2}{c|}{No Precond.} & \multicolumn{2}{c|}{Linear Forest} & \multicolumn{2}{c|}{Greedy} & \multicolumn{2}{c}{Matroid} \\
    $\frac{|\cE|}{|\cV|}$    & $\kappa^2 $ & it & $\kappa^2 $ & it & $\kappa^2 $ & it & $\kappa^2 $ & it \\[1mm]
    \hline
    0.68 & 53.8 & 1624 & 24.9 & 789 & \textbf{1} & \textbf{36} & \textbf{1} & \textbf{36} \\
    1.11 & 257.0 & 6245 & 148.7 & 3782 & \textbf{2} & \textbf{80} & \textbf{2} & \textbf{80} \\
    2.36 & 82.2 & 2061 & 14.1 & 577 & 4 & 143 & \textbf{3} & \textbf{103}\\
    2.98 & 27.7 & 1010 & 12.4 & 484 & 5 & 194 & \textbf{4} & \textbf{137}\\
    4.25 & 24.9 & 800 & 10.1 & 419 & 7 & 254 & \textbf{5} & \textbf{166}\\
    6.48 & 7.7 & 367 & 5.6 & 176 & 10 & 362 & \textbf{1.75} & \textbf{59}\\
    \hline
  \end{tabular}}
  \caption{Comparison of condition number and PDHG iterations for \tg{various} forest strategies on small random graphs ($|\mathcal{V}|=512$) with varying edge to vertex ratio $\frac{m}{n}$.}
  \label{table:ROF_synthetic}
\end{table}

In spite of such favorable properties of matroid partitioning, 
its complexity grows like $\mathcal{O}(|\cE|^3 + |\cE|^2L)$, cf.~\cite{Knu73}, 
making its application prohibitive for large graphs.
\paragraph{Greedy nested forests.} As a remedy, 
we propose the following ``greedy nested forests'' heuristic: given the input graph $\cG$ we successively subtract a spanning forest until no edges remain. The individual subtracted forests form the graph partitioning $\{ \cG_l \}_{l=1}^L$. While this greedy approach does not guarantee a minimal number of forests $L$, the partition still satisfies the assumption of Theorem~\ref{thm:nested}. Indeed, each edge in the forest $\cG_l$ can be represented by a path in $\cG_{l-1}$ \tg{since} adding that edge from $\cG_l$ to $\cG_{l-1}$ would form a cycle \tg{due to} the spanning forest property.

\paragraph{Greedy linear forests.} \tg{Since} the dual update can be computed very efficiently for linear forests we modify the above procedure to yield a linear forest decomposition. Before subtracting the spanning forest from the residual graph, we remove all edges which contain a vertex with degree larger than 2. \tg{In addition}, we check whether it is possible to add any edges from the residual graph to the current linear forest without turning it into a general forest.

In Table~\ref{table:ROF_synthetic} we show the condition number for the different partitioning strategies on small random Erd\H{o}s-Renyi graphs of varying average degree. While the matroid partitioning strategy finds the lowest condition number, both greedy heuristics also lead to a reduction in condition number for most cases. The greedy nested forest heuristic works best for graphs with low edge-to-vertex ratio, while the linear forest heuristic is preferable for dense graphs.

%% file: sections/subproblem.tex
\input{sections/backwardsolver.tex}

In this section we will discuss how the dual update \eqref{eq:ppdhg2} is
computed for our combinatorial preconditioner $T$. 
Assuming that $F^*$ is separable across the subgraphs $\{ \cG_l \}_{l=1}^L$,
\begin{equation}
  F^*(p) = \sum_{l=1}^L F_l^*(p|_{\cE_l}), \notag
\end{equation}
the dual update \eqref{eq:ppdhg2} decomposes into parallel problems:
\iali{
p^{k+1}|_{\cE_l} = &\arg\min_{p \in \bR^{\cE_l}}~\frac{t}{2}\|p-p^k|_{\cE_l}\|_{T_l}^2 +F_l^*(p) \notag \\ 
& -\ip{K\bar u^{k+1}|_{\cE_l}}{p}, \quad \forall ~ 1 \leq l \leq L.
\label{eq:dual_update}
}
Let us now further assume that the individual $\cG_l$ are forests and $F(\cdot) = \norm{\cdot}_1$. 

Expanding the norm $\|\cdot\|_{T_l}^2$ and completing the square in \eqref{eq:dual_update}
leads to the problem
\iali{
p^{k+1}|_{\cE_l} = \arg\min_{ \|p\|_{\infty} \leq 1}~&\frac{1}{2}\|K_l^\top p + f_l \|^2,
\label{eq:dual_problem_tv}
}
with $f_l = - K_l^\top (p^k|_{\cE_l}) - (2u^{k+1} - u^k) / t$. The dual problem to \eqref{eq:dual_problem_tv} is given by
\iali{
  v_l = \arg\min_{u \in \bR^{\cV}} ~ \frac{1}{2} \norm{u - f_l}^2 + \norm{K_l u}_1.
  \label{eq:dual_primal_problem_tv}
}
This further parallelizes into weighted total variation problems on the individual trees in the forest $\cG_l$.
These problems can be handled due to recent advances in direct total-variation solvers; see \cite{davies2001local,dumbgen2009extensions,johnson2013dynamic,Con13,BaSr14,KPR16}. The original taut-string algorithm \cite{davies2001local,dumbgen2009extensions} solves the 1D total-variation problem in $\mathcal{O}(n)$ iterations on a chain. Condat \cite{Con13} proposed an algorithm which has worst-case $\mathcal{O}(n^2)$ complexity but it achieves good performance in practice. Barbero and Sra~\cite{BaSr14} proposed a generalization of the taut-string approach to the case of weighted total variation which runs in $\mathcal{O}(n)$. The approach proposed by Johnson~\cite{johnson2013dynamic} also runs in $\mathcal{O}(n)$ time, works for weighted total variation and has good practical performance. Furthermore, a more memory efficient implementation of Johnson's algorithm generalization to trees was proposed by Kolmogorov~\etal~\cite{KPR16} -- which appears to be \tg{state-of-the-art}.
\begin{algorithm}[t!]
   \caption{PDHG with combinatorial preconditioning for total variation minimization on weighted graphs.}
   \label{alg:pdhg_forest}
\begin{algorithmic}
  \STATE { {\bfseries Input:} $u^0 \in \bR^\cV$, $p^0 \in \bR^\cE$, $\cG=(\cV,\cE,\omega)$.}\\[1mm] 
  \STATE { Compute decomposition of $\cG$ into forests $\{ \cG_l \}_{l=1}^L$. }  
  \STATE { Pick $s, t > 0$ satisfying $st > L$. }\\[1mm]
  \FOR{$k \geq 0$ \textbf{while} not converged}
  \STATE{} \Comment{primal update}\\
  \STATE{$u^{k+1} = \arg\min_u~\frac{s}{2}\|u-u^k\|^2+\ip{K^\top p^k}{u}+G(u)$.}\\
  \STATE{$\bar u^{k+1} = 2 u^{k+1} - u^k$.} \\[1mm]
  \STATE{} \Comment{dual update}\\
  \FOR{each forest $l = 1 \hdots L$}
  \STATE{$f_l = -K_l^\top (p^k|_{\cE_l}) -\bar u^{k+1} / t$.}
  \STATE{Obtain $v_l$ through \eqref{eq:dual_primal_problem_tv} on forest $\cG_l$.}
  \STATE{Obtain $p^{k+1}|_{\cE_l}$ by $K_l^\top p_l^{k+1} = v_l - f_l$.}
  \ENDFOR{}
  \ENDFOR
\end{algorithmic}
\end{algorithm}
We use our implementation of the algorithm proposed by  Kolmogorov~\etal~\cite{KPR16} to find the exact minimizer on each tree. The algorithm computes derivatives of messages $\hat{M}_{i}: \mathbb{R} \rightarrow \mathbb{R}$ and $M_{i,j}: \mathbb{R} \rightarrow \mathbb{R}$ for $(i,j) \in \cE$ and $i, j \in \cV$ in the order from leaves toward the root, which are defined as the following:
\iali{
\hat{M}_{i}{(u)} = \frac{1}{2}(u - f_{l,i})^2 + \sum_{(k,i) \in \cE} M_{k,i}(u), \notag \\
M_{i,j}(u) = \min_{u \in \bR}\left[ \hat{M}_{i}(u) + \omega_{i,j}|u_{j}-u| \right]. \notag
}
The derivatives are denoted as $\hat{m}_{i} := \hat{M}'_{i}$, and $m_{i,j} := M'_{i,j}$. The procedure is summarized in Algorithm~\ref{alg:backward}.

\begin{table*}[t!]
  \centering
  \resizebox{\textwidth}{!}{
  \begin{tabular}{lcc||cc|cc|cc|cc}
    \multicolumn{3}{c||}{Instance} & \multicolumn{2}{c|}{None~\cite{ChPo11}} & \multicolumn{2}{c|}{Diagonal~\cite{PoCh11}} & \multicolumn{2}{c|}{Nested Forest} & \multicolumn{2}{c}{Linear Forest} \\
    name & $\frac{|\cV|}{1024}$ & $\frac{|\cE|}{|\cV|}$ & it & time[s] & it & time[s] & it & time[s] & it & time[s]\\[1mm]
    \hline

    \textbf{synthetic} &&&&&&&&& \\
    rmf-long.n2 & 64 & 2.87 & -- & -- & -- & -- & \textbf{1794} & 62.9 (+0.7) & 5070 & \textbf{24.4} (+1.2) \\
    rmf-wide.n2 & 32 & 2.84 & -- & -- & -- & -- & \textbf{159} & \textbf{2.6} (+0.3) & 23518 & 62.2 (+0.6)\\
    wash-rlg-long.n1024 & 64 & 2.99 & 73309 & 134.1 (+9.4) & 14333 & \textbf{49.7} (+0.0) & \textbf{5848} & 373.9 (+0.8) & 18798 & 108.7 (+1.3) \\[1mm]

    \textbf{bisection} &&&&&&&&& \\
    horse-48112 & 47 & 2.99 & -- & -- & 21593 & 38.3 (+0.0) & \textbf{964} & \textbf{33.9} (+0.3) & 19145 & 90.4 (+2.8) \\
    alue7065-33338 & 33 & 1.61 & -- & -- & 23218 & 21.5 (+0.0) & \textbf{2499} & \textbf{21.4} (+0.17) & 49452 & 99.4 (+1.7)\\[1mm]

    \textbf{stereo}  &&&&&&&&& \\  
    BVZ-venus1* & 162 & 1.99 & 9068 & 20.8 (+7.9) & 3741 & 20.2 (+0.0) & 1111 & 90.7 (+0.9) & \textbf{414} & \textbf{1.9} (+0.1) \\
    BVZ-venus2* & 162 &  1.99 & 10099 & 24.4 (+7.9) & 3124 & 17.2 (+0.0) & 1065 & 88.8 (+0.9) & \textbf{384} & \textbf{1.7} (+0.1)\\
    {KZ2-sawtooth1} & 310 &  2.91 & 96974 & 736.5 (+31.8) & 3468 & 51.7 (+0.0) & \textbf{336} & 102.8 (+3.0) & 526 & \textbf{14.9} (+6.1)\\
    {KZ2-sawtooth2} & 294 & 2.79 & 95849 & 675.3 (+33.1) & 3520 & 48.2 (+0.0) & \textbf{432} & 124.8 (+2.9) &  652 & \textbf{15.7} (+5.8)\\[1mm]

    \textbf{misc vision} &&&&&&&&& \\
    texture\_graph & 9 & 4.76 & 4860 & \textbf{1.6} (+1.45) & 3554 & 1.9 (+0.0) & \textbf{1091} & 9.12 (+0.1) & 1669 & 1.7 (+0.3)\\
    lazybrush-mangagirl* & 579 & 1.99 & -- & -- & -- & -- & 13318 & 5727.3 (+3.3) & \textbf{6330} & \textbf{95.4} (+0.4)\\
    imgseggmm-ferro & 231 & 3.98 & 3594 & 32.2 (+75.9) & 5806 & 95.0 (+0.0) & 786 & 276.8 (+4.4) & \textbf{775} & \textbf{25.6} (+3.1)\\

    \hline
  \end{tabular}}
  \caption{We compare the number of iterations and time required to reach a relative primal dual gap of less than $10^{-10}$ on various graph cut instances. \tr{The time for constructing the preconditioners is shown in brackets (for ``None'' it is the time taken to estimate \tb{$\sigma_{\max}(K)$}). {``--''} indicates that the method failed to reach the desired tolerance within $10^5$ iterations. {``*''} indicates that the graph has grid toplogy.} }
  \label{table:ROF}
\end{table*}
After running Algorithm~\ref{alg:backward} for each forest $\cG_l$, the solution 
$p^{k+1}|_{\cE_l}$ to \eqref{eq:dual_problem_tv} is given by the optimality condition
\begin{equation}
  K_l^\top p^{k+1}|_{\cE_l} = v_l - f_l. \notag
\end{equation}
\tg{Since} each $\cG_l$ is a forest, the corresponding $K_l^\top$ matrix has full column rank which implies that the linear system has a unique solution. Rows of $K_l^\top$ corresponding to leaf nodes in the tree contain exactly one \tg{nonzero} element. Therefore, we solve the linear equation by starting from leaves toward the root. Consistency on the branch nodes is guaranteed \tg{due to} the uniqueness of the solution. 
Even though we only discussed the case of total variation $F(\cdot) = \norm{\cdot}_1$, we remark that for various other choices of $F$ (e.g., Huber penalty) efficient solvers on trees are conceivable.

%% file: sections/backwardsolver.tex
\newcommand{\Comment}[1]{\textit{// #1}}
\newcommand{\LineComment}[1]{\hfill\textit{// #1}}
\begin{algorithm}[tb]
  \caption{Total variation on a forest~\cite[Algorithm~2]{KPR16}.}
  \label{alg:backward}
  \begin{algorithmic}
    \STATE { {\bfseries Input:} Weighted forest $\cG_l = (\cV, \cE_l, \omega_l)$, $f_l \in \bR^\cV$.} \\[1mm]
    \FOR{each tree $\cT = (\cV', \cE', \omega')$ in $\cG_l$}
    \STATE{\Comment{Message passing from leaves to the root.}} 
    \FOR {each $(i,j)\in \cE'$ from leaves to root $r \in \cV'$}
    \STATE{$\hat{m}_{i}(u) = u - f_{l,i} + \sum_{(j,i) \in \cE'} \hat{m}_{j,i}(u)$.}
    \STATE{$m_{i,j}(u) = \text{clip}_{[-\omega'_{i,j}, ~\omega'_{i,j}]} \left( \hat{m}_{i}(u) \right)$.}
    \STATE{find $\lambda^-_{i,j}$, $\lambda^+_{i,j}$ which satisfy both\\
      \qquad $\hat{m}_{i}(\lambda^-_{i,j})=-\omega'_{i,j}$, $\hat{m}_{i}(\lambda^+_{i,j})=\omega'_{i,j}$.}
    \ENDFOR
    \\[1mm]
    \STATE{\Comment{Compute solution on tree.}}
    \STATE{Solve $\hat{m}_{r}(v_{l,r}) = 0$ for $v_{l,r}$.}
    \FOR {each $(i, j) \in \cE'$ from root toward leaves}
    \STATE{$v_{l,i} = \text{clip}_{[\lambda^-_{i,j},\lambda^+_{i,j}]}(v_{l,j})$.}
    \ENDFOR 
    \ENDFOR\\[1mm]

    \STATE { {\bfseries Output:} $v_l \in \bR^\cV$.} 
  \end{algorithmic}
\end{algorithm}

%% file: sections/numerics.tex
\section{Numerical Validation}
The preconditioned PDHG algorithm \eqref{eq:ppdhg1}--\eqref{eq:ppdhg2} for total variation minimization
on weighted graphs is summarized in Algorithm~\ref{alg:pdhg_forest}. We assume that the primal update can be \tr{efficiently computed, e.g.~if $G$ is separable}.

For the experiments we compare the proposed preconditioners to the unpreconditioned variant of PDHG ($S = I$, $T = I$, $s = \norm{K}$, $t = \norm{K}$), the diagonal preconditioners from~\cite{PoCh11} with choice of $\alpha = 1$ and $s = t = 1$. When using the proposed preconditioners we employ the balanced step size choice $s = \sqrt{L}$, $t = \sqrt{L}$.

We implemented all algorithms in MATLAB, whereas time critical parts such as the total variation solver on a tree (Algorithm~\ref{alg:backward}) were implemented in C++.

\subsection{Generalized Fused Lasso}
\label{sec:ROF}
The fused lasso~\cite{tibshirani2005sparsity}, also known as the Rudin-Osher-Fatemi (ROF) model~\cite{Rudin-Osher-Fatemi-92} to the image processing community is readily generalized to graphs:
\begin{equation}
  \min_{u \in \bR^{\cV}} ~ \frac{1}{2} \norm{u - f}^2 + \norm{K u}_1.
  \label{eq:ROF}
\end{equation}
Despite its simplicity, this model \tg{has} a plethora of applications in statistics~\cite{xin2014efficient}, machine learning~\cite{hein2011beyond,bresson2013multiclass} and computer vision~\cite{stuhmer2010real,newcombe2011dtam}, often as a subproblem in sequential convex programming for nonconvex minimization.

We solve \eqref{eq:ROF} using the accelerated PDHG variant (\cite[Algorithm~2]{ChPo11}, $\gamma = 0.25$) \tg{since} the energy \eqref{eq:ROF} is $1$-strongly convex. In Table~\ref{table:ROF_synthetic} we compare the number of iterations required to solve \eqref{eq:ROF} on small random graphs. We stop the algorithm once the relative primal-dual gap drops below $10^{-12}$. It can be observed that there is a clear correlation between the condition number $\kappa(T^{-1/2} K)$ and the number of required iterations, validating the discussion from Section~\ref{sec:pdhg}. The optimal preconditioning based on matroid partitioning performs best. 

We further validate our preconditioner on the maximum flow benchmark~\cite{goldberg11}\footnote{\url{http://www.cs.tau.ac.il/~sagihed/ibfs/}}. It is well known (cf.~\cite{chambolle2009total,PoCh11}) that the minimum cut in a flow network can be obtained by thresholding the minimizing argument of \eqref{eq:ROF}.
We remark that graph cuts can be efficiently \tg{found} by highly specialized combinatorial solvers such as \tg{that in}~\cite{goldberg11}. The point of this experiment is, however, to compare different preconditioners for continuous first-order algorithms on challenging real-world graphs.

In Table~\ref{table:ROF}, we show iterations and run time for the proposed forest preconditioners, the unpreconditioned PDHG algorithm and the diagonal preconditioners~\cite{PoCh11}. \tg{Due to} the size of the graphs, matroid partitioning is intractable, and we resort to the greedy nested forest and linear forest approaches. 
\tg{Combinatorial preconditioning consistently} leads to a significant decrease in iterations. In all except one case, the overall lowest run time is achieved either by linear forest decomposition or greedy nested forest decomposition. 

Despite the huge decrease in \tg{the} total number of outer iterations for the nested forest preconditioning, in some cases, the overall run time is worse than without preconditioning. This motivates the construction of preconditioners like the greedy linear forests. These are clearly suboptimal with respect to condition number, \tg{but} yield a good balance between efficient resolution of the backward step and \tg{the} number of outer iterations. The chains on regular grid achieve the best of both worlds and lead to an improvement \tg{in} runtime of an order of magnitude.

\subsection{Multiclass Segmentation}
\begin{figure}[t!]
  \begin{tabular}{cc}
    \includegraphics[width=0.47\linewidth]{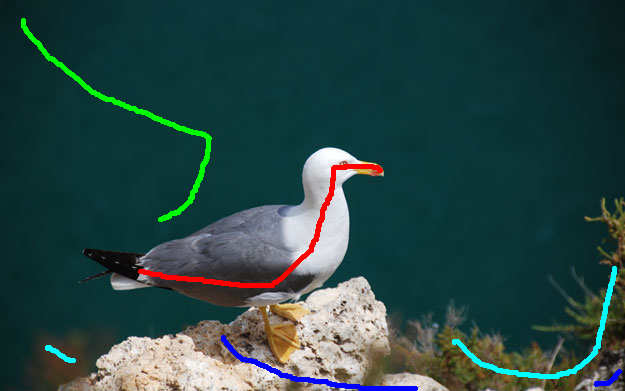}&
    \includegraphics[width=0.47\linewidth]{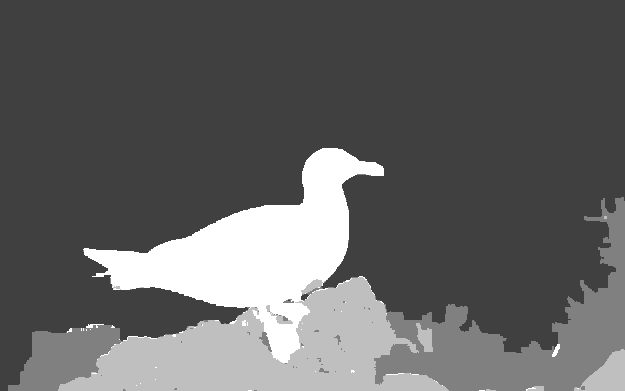}\\
    {\small Input ($625 \times 391$) + Scribbles} & {\small Result ($11$ PDHG iterations)}
  \end{tabular}
  \caption{Interactive image segmentation. Combinatorial preconditioning works particularly well 
  for graphs with underlying grid topology. The multiclass segmentation shown on the right is obtained after only 11 PDHG iterations.}
  \label{fig:clustering}
\end{figure}
As a second example, we consider the multiclass total variation segmentation problem 
\begin{equation}
  \begin{aligned}
    \min_{ u \in \bR^{\tr{|\cV|} \times C} } ~ \sum_{i \in \cV} ~&\delta \{ u_i \in \Delta_C \} + \varepsilon H(u_i) + 
\ip{u_i}{\rho_i} + \norm{K u}_1,
  \end{aligned}
  \notag
\end{equation}
under entropic regularization \tr{$H(u_i) = \sum_{c=1}^C u_{i,c} \log u_{i,c}$. $\delta \{ u_i \in \Delta_C \}$} is the indicator function of the $C$-dimensional unit simplex $\Delta_C \subset \bR^C$.

The above model has various important applications in vision~\cite{zach2008fast,nieuwenhuis2013spatially} and transductive learning~\cite{garcia2014multiclass}. 
As the entropy term renders the objective \tr{($1/\varepsilon$)}-strongly convex we can again use the accelerated PDHG algorithm (\cite[Algorithm~2]{ChPo11} with $\gamma = 1 / \varepsilon$). \tr{We set $\varepsilon = 1$ for all experiments.} Note that the model fits the general saddle point problem \eqref{eq:spp} under the choice $F(p) = \norm{p}_1$, $G(u) = \sum_{i \in \cV} \delta \{ u_i \in \Delta_C \} + \varepsilon H(u_i) + \ip{u_i}{\rho_i}$. To compute the proximal subproblem in \eqref{eq:ppdhg1}
we first observe that 
\tr{$G^*(v) = \sum_{i \in \cV} \varepsilon \log \left( \sum_{c=1}^C \exp \left( (v_{i,c} - \rho_{i,c}) / \varepsilon \right) \right)$.}

Thanks to Moreau's identity we reduce \eqref{eq:ppdhg1} to the proximal evaluation of $G^*$, for which a few Newton iterations suffice.
In Table~\ref{table:clustering} we compare the performance of our combinatorial preconditioners on two of the aforementioned applications. For the transductive learning scenario, we follow the procedure described in \cite{garcia2014multiclass} to generate a $k$-nearest neighbour graph ($k=10$) on the synthetic ``three moons'' data set. As in~\cite{garcia2014multiclass}, the data term \tr{$\rho \in \bR^{|\cV| \times C}$} specifies the correct labels for $5\%$ of the points. We report a similar final accuracy ($98.9\%$) as the authors of~\cite{garcia2014multiclass}. As seen in Table~\ref{table:clustering}, the nested forest preconditioning \tr{performs best.} 

For the second application, we consider interactive image segmentation~\cite{nieuwenhuis2013spatially}. Following that paper, we compute the data term $\rho$ from user scribbles (see Fig.~\ref{fig:clustering}). The \tr{weights} are chosen based on the input image $I \in \bR^{\tr{|\cV|} \times 3}$ as
\begin{equation}
  \omega_{i,j} = \exp \left( -\xi \norm{I_i - I_j}^2 \right), ~ \forall (i, j) \in \mathcal{E}. \notag
\end{equation}
This essentially \tr{encourages the segmentation boundary to coincide with image discontinuities. We use a fixed scale parameter $\xi = 0.1$ in all experiments.} The performance of the different preconditioning strategies is shown in Table~\ref{table:clustering}. Due to the underlying grid topology the natural linear forest decomposition into chains can be employed, which outperforms the other preconditioners and PDHG without preconditioning by an order of magnitude. While the nested forest preconditioner is competitive w.r.t. iterations, the proximal subproblem on the large spanning tree is very expensive. In contrast, the subproblem on the short chains can be computed \tr{efficiently}.

%% file: sections/conclusion.tex
\begin{table}
  \centering
  \resizebox{1\linewidth}{!}{
  \begin{tabular}{l|c|c|c|c}
     & \multicolumn{1}{c|}{None~\cite{ChPo11}} & \multicolumn{1}{c|}{Diag.~\cite{PoCh11}} & \multicolumn{1}{c|}{Nest. Forest} & \multicolumn{1}{c}{Lin. Forest} \\
    name & it (time[s]) & it (time[s]) & it (time[s]) & it (time[s])\\
    \hline
    & & & &\\[-3mm]
    3MOONS & 333 (2.8) & 474 (8.2) & \textbf{88} (\textbf{1.1}) & 303 (3.1) \\[1.5mm]
    icgbench-1* & 113 (47.0) & 131 (41.3) & 52 (41.9) & \textbf{13} (\textbf{3.1})  \\
    icgbench-2* & 124 (54.7) & 159 (54.1) & 58 (47.9) & \textbf{11} (\textbf{2.6}) \\
    icgbench-3* & 78 (29.1) & 95 (25.9) & 42 (27.2) & \textbf{9} (\textbf{1.6})\\[1mm]
    \hline
  \end{tabular} }
  \caption{We compare iterations and time (in brackets) required for PDHG under various choices of preconditioner to reach a relative primal-dual gap less than $5 \cdot 10^{-4}$. On general graphs, the greedy nested forests perform well while on regular grids (indicated with ``*'') the linear forest decomposition into chains works best.}
  \label{table:clustering}
\end{table}

\section{Conclusion}
We proposed a novel combinatorial preconditioner for proximal
algorithms on weighted graphs. The preconditioner is driven by a
disjoint decomposition of the edge set into forests. Our theoretical
analysis provides conditions under which such a decomposition achieves
an optimal condition number. Furthermore, we \tr{have shown} how provably optimal
preconditioners can be obtained: on \tb{two-dimensional} regular grids by a splitting into
horizontal and vertical chains, on general graphs by means of matroid
partitioning. We additionally proposed two fast heuristics to
construct reasonable preconditioners on large scale graphs.  We \tr{demonstrated} how the resulting scaled proximal evaluations can be carried out
by means of an efficient message passing algorithm on trees.  In an extensive numerical
evaluation we \tr{confirmed} \tr{practical gains} of preconditioning \tr{in terms of}
overall \tr{runtime as well as outer iteration numbers}. 

\subsubsection*{Acknowledgements} We thank Thomas Windheuser for fruitful discussions on combinatorial preconditioning. We gratefully acknowledge the support of the ERC Consolidator Grant 3D Reloaded.